\documentclass{amsproc}

\usepackage{amsrefs}

\newtheorem{theorem}{Theorem}[section]
\newtheorem{corollary}[theorem]{Corollary}
\newtheorem{lemma}[theorem]{Lemma}
\newtheorem{proposition}[theorem]{Proposition}

\theoremstyle{definition}
\newtheorem{definition}[theorem]{Definition}

\newtheorem{conjecture}[theorem]{Conjecture}

\theoremstyle{remark}
\newtheorem{remark}[theorem]{Remark}

\numberwithin{equation}{section}

\begin{document}

\title{Ulrich Bundles on Veronese Surfaces}


\author{Emre Coskun}
\address{Middle East Technical University, 06800, Ankara, TURKEY}
\curraddr{}
\email{emcoskun@metu.edu.tr}
\thanks{The first author has been supported by TUBITAK project 114F116.}

\author{Ozhan Genc}
\address{Middle East Technical University, 06800, Ankara, TURKEY}
\curraddr{Mathematics Research and Teaching Group, Middle East Technical University, Northern Cyprus Campus, KKTC, Mersin 10, Turkey}
\email{ozhangenc@gmail.com}
\thanks{}


\date{\today}

\begin{abstract}
We prove that every Ulrich bundle on the Veronese surface has a resolution in terms of twists of the trivial bundle over $\mathbb{P}^{2}$. Using this classification, we prove existence results for stable Ulrich bundles over $\mathbb{P}^{k}$ with respect to an arbitrary polarization $dH$.
\end{abstract}

\maketitle

\section{Introduction}
We recall the following characterization of Ulrich bundles (\cite{CKM2013}*{Proposition 2.3}).
\begin{definition}
Let $\mathcal{E}$ be a vector bundle of rank r on smooth projective variety $X \subset \mathbb{P}^N$ of dimension $k \geq 2$ and of degree $d$. Then $\mathcal{E}$ is Ulrich if and only if its Hilberts polynomial is $dr \binom{t+k}{k}$ and $H^{q}(X, \mathcal{E}(t))=0$ for $0 < q < l$ and $t \in \mathbb{Z}$ (this is known as the \textit{arithmetically Cohen-Macaulay (ACM) condition}).
\end{definition}
\begin{remark}
Note that the embedding $X \subset \mathbb{P}^N$ is part of the definition.
\end{remark}

The existence of Ulrich bundles on smooth projective varieties is related to a number of geometric questions. For instance, the existence of rank 1 or rank 2 Ulrich bundles on a hypersurface is related to the representation of that hypersurface as a determinant or Pfaffian (\cite{aB2000}). Another question of interest is the Minimal Resolution Conjecture (MRC) (\cite{aL1993}, \cite{FMP2003}). Also, in \cite{ES2011}, it is proved that the cone of cohomology tables of vector bundles on a $k$-dimensonal variety $X \subset \mathbb{P}^{N}$ is the same as the cone of cohomology tables of vector bundles on $\mathbb{P}^{k}$ if and only if there exists an Ulrich bundle on $X$.

It was conjectured in \cite{ESW2003} that there exist Ulrich bundles on any variety. Although it is known that all projective curves, hypersurfaces, Veronese varieties, abelian surfaces and many K3 surfaces admit Ulrich bundles, such a general existence result is not known.

One problem that has attracted a lot of attention recently is the existence of stable Ulrich bundles with given rank and Chern classes. Stable Ulrich bundles are particularly interesting as they are the building blocks of all Ulrich bundles. Every Ulrich bundle is semistable, and the Jordan-H\"{o}lder factors are stable Ulrich bundles by Lemma \ref{lem:ulrich_stability}.

Since the definition of Ulrich bundles depends on the polarization of the underlying variety, the following question arises.

\begin{center}
\textit{How do the different polarizations of the underlying variety affect the existence and classification of Ulrich bundles?}
\end{center}

In light of this question, in this work, we study the existence of stable Ulrich bundles of arbitrary rank on $\mathbb{P}^{2}$ with respect to the polarization $dH$ for any positive integer $d$.

\begin{definition}
For a positive integer $d$, the Veronese map of degree $d$ is the map $\mathbb{P}^{2} \to \mathbb{P}^{\binom{d+2}{2}-1}$ given by the very ample line bundle $\mathcal{O}_{\mathbb{P}^{2}}(dH)$. Note that its image is a surface of degree $d^{2}$. We call this surface the Veronese surface and denote it by $(\mathbb{P}^{2},dH)$.
\end{definition}

First, we investigate the Ulrich line bundles and we prove that Ulrich line bundles exist only for $d=1$; in that case there is only one Ulrich line bundle and it is isomorphic to $\mathcal{O}_{\mathbb{P}^{2}}$. We also prove that, for $d=1$, any rank $r$ Ulrich bundle is isomorphic to a direct sum of $r$ copies of $\mathcal{O}_{\mathbb{P}^{2}}$. From now on, we assume that $d \geq 2$.

The next step is to investigate stable Ulrich bundles of rank $r \geq 2$. First, we establish necessary conditions for the existence of such bundles. We prove that if $\mathcal{E}$ is a rank $r$ Ulrich bundle on $(\mathbb{P}^{2},dH)$ then its first Chern class $c_{1}(\mathcal{E})=\frac{3r(d-1)}{2}$. Using this, we can see directly that if $d$ is even, then $r$ must be even. Moreover, using Beilinson's theorem (\cite{OSS2011}*{Theorem 3.1.3, p.240}), we prove that $\mathcal{E}$ has a minimal free resolution
\[
 0 \to \mathcal{O}_{\mathbb{P}^{2}}^{\oplus\frac{r}{2}(d-1)}(d-2) \stackrel{f}{\rightarrow} \mathcal{O}_{\mathbb{P}^{2}}^{\oplus\frac{r}{2}(d+1)}(d-1) \to \mathcal{E} \to 0.
\]

We then construct Ulrich bundles using these necessary conditions. For $d=2$, we prove that there exists a unique rank 2 Ulrich bundle, which was also given in \cite{ESW2003}. Also, we prove that higher rank Ulrich bundles for $d=2$ are direct sums of this unique rank 2 Ulrich bundle.

Finally, for $d>2$, we analyze the cases where $d$ is even and odd separately. For even $d$, we prove the existence of stable Ulrich bundles for any (even) rank using the Casanellas-Hartshorne method \cite{CHGS2012} by consecutive extensions of rank 2 Ulrich bundles. For odd $d$, we have to assume that rank 3 Ulrich bundles exist. We make a conjecture for this assumption, since using \textit{Macaulay2} we demonstrate that rank 3 Ulrich bundles exist on $(\mathbb{P}^{2},dH)$ for odd degrees up to $d=43$. Then, assuming the existence of rank 3 Ulrich bundles and the Casanellas-Hartshorne method, we prove that there exist stable Ulrich bundles for any rank $r \geq 2$.

\subsection{Conventions}
We work over an algebraically closed base field of characteristic 0. All cohomologies are over $\mathbb{P}^2$.

\section{Preliminaries}
Recall that, for a positive integer $d$, the Veronese surface is denoted by $(\mathbb{P}^2, dH)$.

The following result might be termed \emph{Ulrich duality,} and is analogous to \cite{CKM2013}*{Proposition 2.11}. (Note that we cannot apply \cite{CKM2013}*{Proposition 2.11} directly in this situation, since the canonical class of $\mathbb{P}^2$ may not be a multiple of the polarization $dH$.)
\begin{proposition}\label{prop:Ulrich_duality}
Let $\mathcal{E}$ be an Ulrich bundle of rank $r$ on $(\mathbb{P}^2, dH)$. Then $\mathcal{E}^{\vee}(3d-3)$ is also Ulrich.
\end{proposition}
\begin{proof}
By \cite{CKM2013}*{Proposition 2.3}, we must prove that $\mathcal{E}^{\vee}(3d-3)$ is ACM and that its Hilbert polynomial is $d^2 r \binom{t+2}{2}$.

The fact that $\mathcal{E}^{\vee}(3d-3)$ is ACM follows from Serre duality and the fact that $\mathcal{E}$ is ACM:
\[
 H^1(\mathcal{E}^{\vee}(3d-3)(dt)) = H^1(\mathcal{E}(-3d-dt)) = H^1(\mathcal{E}((-3-t)d)) = 0.
\]
Let us now compute the Hilbert polynomial of $\mathcal{E}^{\vee}(3d-3)$.
\begin{align*}
 \chi(\mathcal{E}^{\vee}(3d-3)(dt)) &= \sum_{i=0}^{2} (-1)^i h^i(\mathcal{E}^{\vee}(3d-3)(dt)) \\
 &= \sum_{i=0}^{2} (-1)^i h^{2-i}(\mathcal{E}(-3d-dt)) \\
 &= \chi(\mathcal{E}(d(-3-t)) \\
 &= d^2 r \binom{-3-t+2}{2} \\
 &= d^2 r \binom{t+2}{2}.
\end{align*}
\end{proof}

\begin{proposition}\label{prop:c_1}
Let $\mathcal{E}$ be an Ulrich bundle of rank $r$ on $(\mathbb{P}^2, dH)$. Then $c_1(\mathcal{E})=3r(d-1)/2$.
\end{proposition}
\begin{proof}
If $\mathcal{E}$ is an Ulrich bundle of rank $r$ on $(\mathbb{P}^2, dH)$, its Hilbert polynomial is $d^2 r \binom{t+2}{2}$ by \cite{CKM2013}*{Proposition 2.3}. The proof involves a simple application of the Riemann-Roch theorem for vector bundles on surfaces. By \cite{wF1984}*{Ex 3.2.2}, we have
\begin{align*}
 c_1(\mathcal{E}(dt)) &= c_1(\mathcal{E}) + rdt \\
 c_2(\mathcal{E}(dt)) &= c_2(\mathcal{E}) + (r-1) dt c_1(\mathcal{E}) + \binom{r}{2} (dt)^2.
\end{align*}
Now, by the Riemann-Roch theorem, we have
\begin{align*}
 \chi(\mathcal{E}(dt)) &= \frac{c_1(\mathcal{E}(dt)) (c_1(\mathcal{E}(dt)) - K_{\mathbb{P}^2})}{2} - c_2(\mathcal{E}(dt)) + r \chi(\mathcal{O}_{\mathbb{P}^2}) \\
 &= \frac{(c_1(\mathcal{E})+rdt)(c_1(\mathcal{E})+rdt+3)}{2} - c_2(\mathcal{E}) -  (r-1) dt c_1(\mathcal{E}) - \binom{r}{2}(dt)^2 + r.
\end{align*}
Comparing the coefficients of $t$ in the above expression and $d^2 r \binom{t+2}{2}$, we obtain
\[
 \frac{3d^2r}{2} = \frac{rd(2c_1(\mathcal{E})+3)}{2}-(r-1)dc_1(\mathcal{E}).
\]
Solving this equation for $c_1(\mathcal{E})$, we obtain
\[
 c_1(\mathcal{E})=\frac{3r(d-1)}{2}.
\]
\end{proof}
\begin{corollary}\label{cor:rank_and_degree}
If $d$ is even, the rank $r$ of an Ulrich bundle on $(\mathbb{P}^2, dH)$ must also be even.
\end{corollary}

\begin{lemma}\label{lemma:h2_vanishing}
For Ulrich bundles $\mathcal{E}$ and $\mathcal{F}$ on $(\mathbb{P}^2, dH)$, we have $h^2(\mathcal{E} \otimes \mathcal{F}^{\vee})=0$.
\end{lemma}
\begin{proof}
\begin{align*}
 h^2(\mathcal{E} \otimes \mathcal{F}^{\vee}) &= h^0(\mathcal{E}^{\vee} \otimes \mathcal{F} \otimes \mathcal{O}_{\mathbb{P}^2}(-3)) \\
 &= hom(\mathcal{E}(3),\mathcal{F}) = 0.
\end{align*}
The last equality follows because as an Ulrich bundle, $\mathcal{E}$ is semistable by \cite{CKM2013}*{Proposition 2.6}; and the reduced Hilbert polynomial of $\mathcal{E}(3)$ is strictly greater than the reduced Hilbert polynomial of $\mathcal{E}$. The desired vanishing now follows from \cite{HL2010}*{Proposition 1.2.7}.
\end{proof}
We have the following corollary, which will be useful later in dimension calculations.
\begin{corollary}\label{cor:h1_simple}
For a simple Ulrich bundle $\mathcal{E}$ of rank $r$ on $(\mathbb{P}^2, dH)$, we have $h^1(\mathcal{E} \otimes \mathcal{E}^{\vee})=\frac{1}{4}(4+r^2(d^2-5)))$.
\end{corollary}
\begin{proof}
This is a simple calculation involving the Chern classes. The result follows from the fact that $h^0(\mathcal{E} \otimes \mathcal{E}^{\vee})=1$, Proposition \ref{prop:c_1}, Lemma \ref{lemma:h2_vanishing} and \cite{CHGS2012}*{Proposition 2.12}.
\end{proof}

The following result is crucial for the application of the method of Casanellas and Hartshorne.
\begin{lemma}\label{lem:ulrich_stability}
An Ulrich bundle is Gieseker-semistable. For Ulrich bundles Gieseker-stability and $\mu$-stability are equivalent. The Jordan-H\"{o}lder factors of an Ulrich bundle are stable Ulrich bundles. Therefore, any Ulrich bundle is obtained by consecutive extensions of stable Ulrich bundles.
\end{lemma}
\begin{proof}
See \cite{CKM2012}*{Proposition 2.11, Corollary 2.16} and \cite{CHGS2012}*{Theorem 2.9 (c)}.
\end{proof}

We shall use the following result later.
\begin{lemma}\label{lemma:various_vanishings}
For an Ulrich bundle $\mathcal{E}$ of rank $r$ on $(\mathbb{P}^2, dH)$, we have:
\begin{itemize}
 \item $h^0(\mathcal{E}(-d))=h^1(\mathcal{E}(-d))=h^2(\mathcal{E}(-d))=0$,
 \item $h^0(\mathcal{E}(-d+1))=\frac{r}{2}(d+1)$, $h^1(\mathcal{E}(-d+1))=h^2(\mathcal{E}(-d+1))=0$,
 \item $h^0(\mathcal{E}(-d+2))=r(d+2)$ and $h^1(\mathcal{E}(-d+2))=0$.
\end{itemize}
\end{lemma}
\begin{proof}
For the first item, note that $\mathcal{E}(-d)$ is in fact the $(-1)$-twist of the Ulrich bundle $\mathcal{E}$ according to the polarization $dH$, hence $h^0(\mathcal{E}(-d))=0$. Since $\mathcal{E}$ is ACM, $h^1(\mathcal{E}(-d))=0$. The Euler characteristic of $\mathcal{E}(-d)$ is 0 by \cite{CKM2013}*{Proposition 2.3}. It follows that $h^2(\mathcal{E}(-d))=0$.

For the rest of the proof, let $L \subset \mathbb{P}^2$ be a line.

Consider the restriction $\mathcal{E}|_{L}$. By a well-known result of Grothendieck, we have
\[
 \mathcal{E}|_{L} = \bigoplus_{i=1}^{r} \mathcal{O}_{\mathbb{P}^1}(a_i)
\]
for some integers $a_i$. Using Proposition \ref{prop:c_1}, we have
\[
 deg(\mathcal{E}|_{L}) = \sum_{i=1}^{r} a_i = \frac{3r}{2}(d-1).
\]
Consider the short exact sequence
\[
 0 \to \mathcal{O}_{\mathbb{P}^2}(-1) \to \mathcal{O}_{\mathbb{P}^2} \to \mathcal{O}_{L} \to 0.
\]
Tensoring with $\mathcal{E}(-d)$, we have
\[
 0 \to \mathcal{E}(-1-d) \to \mathcal{E}(-d) \to \mathcal{E}|_{L}(-d) \to 0.
\]
The associated long exact sequence is
\begin{align*}
 0 &\to H^0(\mathcal{E}(-1-d)) &\to H^0(\mathcal{E}(-d)) &\to H^0(\mathcal{E}|_{L}(-d)) \\
 &\to H^1(\mathcal{E}(-1-d)) &\to H^1(\mathcal{E}(-d)) &\to H^1(\mathcal{E}|_{L}(-d)) \\
 &\to H^2(\mathcal{E}(-1-d)) &\to H^2(\mathcal{E}(-d)) &\to 0.
\end{align*}
Note that the middle term in each line is 0 by the previous item.

Now,
\begin{align*}
 h^2(\mathcal{E}(1-d)) &= h^0(\mathcal{E}^{\vee}(d-4)) \\
 &= h^0(\mathcal{E}^{\vee}(3d-3)(-1-2d)) \\
 &= 0,
\end{align*}
where we have used the facts that, by Proposition \ref{prop:Ulrich_duality}, $\mathcal{E}^{\vee}(3d-3)$ is an Ulrich bundle, and that any twist of an Ulrich bundle by an integer smaller than or equal to $-d$ has no global sections. Hence, the long exact sequence above gives us $h^1(\mathcal{E}|_{L}(-d))=h^1(\bigoplus_{i=1}^{r} \mathcal{O}_{\mathbb{P}^1}(a_i-d))=0$. It follows that $a_i > d-2$ for all $i$.

We now prove the second item.

Consider the short exact sequence
\[
 0 \to \mathcal{E}(-d) \to \mathcal{E}(-d+1) \to \mathcal{E}|_{L}(-d+1) \to 0.
\]
The associated long exact sequence is
\begin{align*}
 0 &\to H^0(\mathcal{E}(-d)) &\to H^0(\mathcal{E}(-d+1)) &\to H^0(\mathcal{E}|_{L}(-d+1)) \\
 &\to H^1(\mathcal{E}(-d)) &\to H^1(\mathcal{E}(-d+1)) &\to H^1(\mathcal{E}|_{L}(-d+1)) \\
 &\to H^2(\mathcal{E}(-d)) &\to H^2(\mathcal{E}(-d+1)) &\to 0.
\end{align*}
Note that the first term in each line is 0 by the previous item. Hence, we immediately obtain $h^2(\mathcal{E}(-d+1))=0$.

We have
\begin{align*}
 h^1(\mathcal{E}(-d+1)) &= h^1(\mathcal{E}|_{L}(-d+1)) \\
 &= h^1(\bigoplus_{i=1}^{r} \mathcal{O}_{\mathbb{P}^1}(a_i-d+1)) \\
 &= 0,
\end{align*}
since $a_i-d+1 > -1$ for all $i$.

Next,
\begin{align*}
 h^0(\mathcal{E}(-d+1)) &= h^0(\mathcal{E}|_{L}(-d+1)) \\
 &= h^0(\bigoplus_{i=1}^{r} \mathcal{O}_{\mathbb{P}^1}(a_i-d+1)) \\
 &= \sum_{i=1}^{r} (a_i-d+2) \\
 &= \frac{3r}{2}(d-1) - rd + 2r \\
 &= \frac{r}{2}(d+1).
\end{align*}
This finishes the proof of the second item.

We finish by proving the third item.

Consider the short exact sequence
\[
 0 \to \mathcal{E}(-d+1) \to \mathcal{E}(-d+2) \to \mathcal{E}|_{L}(-d+2) \to 0.
\]
The associated long exact sequence is
\begin{align*}
 0 &\to H^0(\mathcal{E}(-d+1)) &\to H^0(\mathcal{E}(-d+2)) &\to H^0(\mathcal{E}|_{L}(-d+2)) \\
 &\to H^1(\mathcal{E}(-d+1)) &\to H^1(\mathcal{E}(-d+2)) &\to H^1(\mathcal{E}|_{L}(-d+2)) \\
 &\to H^2(\mathcal{E}(-d+1)) &\to H^2(\mathcal{E}(-d+2)) &\to 0.
\end{align*}
Note again that the first terms in the second and third lines are 0 by the previous item.

We have
\begin{align*}
 h^1(\mathcal{E}(-d+2)) &= h^1(\mathcal{E}|_{L}(-d+2)) \\
 &= h^1(\bigoplus_{i=1}^{r} \mathcal{O}_{\mathbb{P}^1}(a_i-d+2)) \\
 &= 0,
\end{align*}
since $a_i-d+2 > 0$ for all $i$.

Finally, from the long exact sequence above, we have
\begin{align*}
 h^0(\mathcal{E}(-d+2)) &= h^0(\mathcal{E}(-d+1)) + h^0(\mathcal{E}|_{L}(-d+2)) \\
 &= \frac{r}{2}(d+1) + h^0(\bigoplus_{i=1}^{r} \mathcal{O}_{\mathbb{P}^1}(a_i-d+2)) \\
 &= \frac{r}{2}(d+1) + \sum_{i=1}^{r} (a_i-d+3) \\
 &= \frac{r}{2}(d+1) + \sum_{i=1}^{r} a_i -rd + 3r \\
 &= r(d+2).
\end{align*}

This finishes the proof.
\end{proof}

\section{Ulrich Line Bundles}

\begin{proposition}\label{prop:ulrich_line_bundles}
The only Ulrich line bundle on $(\mathbb{P}^2, H)$ is $\mathcal{O}_{\mathbb{P}^2}$ and all higher rank Ulrich bundles are simply direct sums of copies of this unique Ulrich line bundle. For $d \geq 2$, there are no Ulrich line bundles on $(\mathbb{P}^2, dH)$.
\end{proposition}
\begin{proof}
The first statement of the theorem follows trivially from \cite{CKM2013}*{Proposition 2.2 (iv)}.

Now let $\mathcal{O}_{\mathbb{P}^2}(t_0)$ be an Ulrich line bundle on $(\mathbb{P}^2,dH)$. We compute its Hilbert polynomial.
\begin{align*}
 \chi(\mathcal{O}_{\mathbb{P}^2}(t_0+dt)) &= \binom{t_0+dt+2}{2} \\
 &= \frac{(dt+t_0+2)(dt+t_0+1)}{2} \\
 &= \frac{d^2 t^2 + d(2t_0+3)t + t_0^2+3t_0+2}{2}.
\end{align*}
By \cite{CKM2013}*{Proposition 2.3}, the Hilbert polynomial of the Ulrich line bundle $\mathcal{O}_{\mathbb{P}^2}(t_0)$ must be equal to
\[
 d^2 \binom{t+2}{2}.
\]
Comparing coefficients, we obtain the equations
\begin{align*}
 3d^2 &= d(2t_0+3) \\
 2d^2 &= t_0^2 + 3t_0 + 2.
\end{align*}
Solving this system, we obtain $d=1$ and $t_0=0$.
\end{proof}

\section{The Beilinson Spectral Sequence And The Cohomology Table}
In this section, we derive the cohomology table used in the Beilinson spectral sequence and we use it to prove that every Ulrich bundle occurs as the right-hand term of a suitable short exact sequence. We also prove a converse result; namely if an injective morphism of sheaves satisfies certain conditions, then its cokernel is an Ulrich bundle.

Recall that for a coherent sheaf $\mathcal{F}$ on $\mathbb{P}^2$, the Beilinson spectral sequence has the form
\[
 E_1^{p,q} := H^q(\mathcal{F} \otimes \Omega^{-p}(-p)) \otimes \mathcal{O}_{\mathbb{P}^2}(p) \Rightarrow E^{p+q} = \left\{ \begin{array}{cl}
 \mathcal{F} & \textrm{if }p+q=0\\
 0 & \textrm{otherwise.}\end{array}\right.
\]
(Here, $\Omega$ denotes the cotangent bundle of $\mathbb{P}^2$.)
\begin{remark}
There is another form of the Beilinson spectral sequence, with $\Omega^{-p}(-p)$ and $\mathcal{O}_{\mathbb{P}^2}(p)$ interchanged; but we shall not use it.
\end{remark}
The entries on the first page of this spectral sequence can only be nonzero for $p \in [-2,0]$ and $q \in [0,2]$.
\begin{theorem}
Let $\mathcal{E}$ be an Ulrich bundle of rank $r$ on $(\mathbb{P}^2, dH)$ and let $\mathcal{F}=\mathcal{E}(-d+1)$. Then we have the following table for the values of $h^q(\mathcal{F} \otimes \Omega^{-p}(-p))$:
\begin{table}[!h]\label{table}
\centering
\bgroup
\def\arraystretch{1.5}
\begin{tabular}{cccc}
\cline{1-3}
\multicolumn{1}{|c|}{0} & \multicolumn{1}{c|}{0} & \multicolumn{1}{c|}{0} & q=2 \\ \cline{1-3}
\multicolumn{1}{|c|}{0} & \multicolumn{1}{c|}{0} & \multicolumn{1}{c|}{0} & q=1 \\ \cline{1-3}
\multicolumn{1}{|c|}{0} & \multicolumn{1}{c|}{$\frac{r}{2}(d-1)$} & \multicolumn{1}{c|}{$\frac{r}{2}(d+1)$} & q=0 \\ \cline{1-3}
p=-2                    & p=-1                   & p=0                    &    
\end{tabular}
\egroup
\caption{The cohomology table for $\mathcal{F} \otimes \Omega^{-p}(-p)$}
\end{table}
\end{theorem}
\begin{proof}
We start with $p=-2$. Note that $\Omega^{2}(2)=\mathcal{O}_{\mathbb{P}^2}(-1)$ and $\mathcal{F} \otimes \Omega^{-p}(-p) = \mathcal{E}(-d)$. All cohomologies vanish by Lemma \ref{lemma:various_vanishings}, which establishes the left column of the table.

Now consider the short exact sequence
\[
 0 \to \Omega \to \mathcal{O}_{\mathbb{P}^2}(-1)^{3} \to \mathcal{O}_{\mathbb{P}^2} \to 0.
\]
Tensoring this short exact sequence with $\mathcal{F}(1)=\mathcal{E}(-d+2)$, we obtain
\[
 0 \to \mathcal{E}(-d+2) \otimes \Omega \to \mathcal{E}(-d+1)^3 \to \mathcal{E}(-d+2) \to 0.
\]
The associated long exact sequence is
\begin{align*}
 0 &\to H^0(\mathcal{E}(-d+2) \otimes \Omega) &\to H^0(\mathcal{E}(-d+1)^3) &\to H^0(\mathcal{E}(-d+2)) \\
 &\to H^1(\mathcal{E}(-d+2) \otimes \Omega) &\to H^1(\mathcal{E}(-d+1)^3) &\to H^1(\mathcal{E}(-d+2)) \\
 &\to H^2(\mathcal{E}(-d+2) \otimes \Omega) &\to H^2(\mathcal{E}(-d+1)^3) &\to H^2(\mathcal{E}(-d+2)) \to 0.
\end{align*}

For the remaining entries, we proceed separately.

\underline{$(p,q)=(-1,2):$} $h^2(\mathcal{E}(-d+1))=h^1(\mathcal{E}(-d+2))=0$ by Lemma \ref{lemma:various_vanishings}, which forces $h^2(\mathcal{E}(-d+2) \otimes \Omega)=h^2(\mathcal{F} \otimes \Omega^{1}(1))=0$.

\underline{$(p,q)=(0,2):$} $h^2(\mathcal{F} \otimes \Omega^{0}(0))=h^2(\mathcal{E}(-d+1))=0$ by Lemma \ref{lemma:various_vanishings}.

\underline{$(p,q)=(0,1):$} $h^1(\mathcal{F} \otimes \Omega^{0}(0))=h^1(\mathcal{E}(-d+1))=0$ by Lemma \ref{lemma:various_vanishings}.

\underline{$(p,q)=(0,0):$} $h^0(\mathcal{F} \otimes \Omega^{0}(0))=h^0(\mathcal{E}(-d+1))=\frac{r}{2}(d+1)$ by Lemma \ref{lemma:various_vanishings}.

\underline{$(p,q)=(-1,1):$} Assume that $h^1(\mathcal{F} \otimes \Omega^{1}(1)) =a \neq 0$. Then the spectral sequence gives us a surjection
\[
 \mathcal{F}=\mathcal{E}(-d+1) \to \mathcal{O}_{\mathbb{P}^2}(-1)^a.
\]
Twisting by $-1$, we get a surjection
\[
 \mathcal{E}(-d) \to \mathcal{O}_{\mathbb{P}^2}(-2)^a.
\]
But $\mathcal{E}(-d)$ is semistable by \cite{CKM2012}*{Proposition 2.11}; and since the reduced Hilbert polynomial of $\mathcal{E}(-d)$ is not smaller than or equal to (with respect to the lexicographical ordering) the reduced Hilbert polynomial of $\mathcal{O}_{\mathbb{P}^2}(-2)^a$, we conclude that $a=0$.

\underline{$(p,q)=(-1,0):$} Since $h^1(\mathcal{E}(-d+2) \otimes \Omega)=h^1(\mathcal{F} \otimes \Omega^{1}(1))=0$ by the previous item, the long exact sequence above gives us
\[
 h^0(\mathcal{F} \otimes \Omega^{1}(1)) = h^0(\mathcal{E}(-d+2) \otimes \Omega) = 3 h^0(\mathcal{E}(-d+1)) - h^0(\mathcal{E}(-d+2)).
\]
By Lemma \ref{lemma:various_vanishings}, we have $h^0(\mathcal{E}(-d+1)) = \frac{r}{2}(d+1)$ and $h^0(\mathcal{E}(-d+2)) = r(d+2)$. Calculating, we find $h^0(\mathcal{F} \otimes \Omega^{1}(1)) = \frac{r}{2}(d-1)$. This completes the table.
\end{proof}
\begin{corollary}\label{cor:short_exact_sequence}
Let $\mathcal{E}$ be an Ulrich bundle of rank $r$ on $(\mathbb{P}^2, dH)$. Then we have
\[
 0 \to \mathcal{O}_{\mathbb{P}^2}^{\frac{r}{2}(d-1)}(d-2) \to \mathcal{O}_{\mathbb{P}^2}^{\frac{r}{2}(d+1)}(d-1) \to \mathcal{E} \to 0.
\]
\end{corollary}
\begin{proof}
From Beilinson's spectral sequence and Table 1, we obtain
\[
 0 \to \mathcal{O}_{\mathbb{P}^2}^{\frac{r}{2}(d-1)}(-1) \to \mathcal{O}_{\mathbb{P}^2}^{\frac{r}{2}(d+1)} \to \mathcal{F}=\mathcal{E}(-d+1) \to 0.
\]
Twisting by $d-1$, we obtain the desired result.
\end{proof}

We now prove a converse result.
\begin{theorem}\label{thm:ulrich_from_map}
Let $\mathcal{E}$ be a vector bundle of rank $r$ on $\mathbb{P}^2$ admitting a resolution 
\[
 0 \to \mathcal{O}_{\mathbb{P}^2}^{\frac{r}{2}(d-1)}(d-2) \to \mathcal{O}_{\mathbb{P}^2}^{\frac{r}{2}(d+1)}(d-1) \to \mathcal{E} \to 0,
\]
such that $H^{1}(\mathcal{E}(-2d))=H^{1}(\mathcal{E}(-3d))=\ldots=H^{1}(\mathcal{E}(-\alpha d))=0$ for $\alpha=\left \lceil{\frac{r+2}{2}}\right \rceil$. Then $\mathcal{E}$ is an Ulrich bundle on $(\mathbb{P}^2,dH)$.
\end{theorem}

\begin{proof}
Since $\mathcal{E}$ admits the given resolution, we have
\begin{eqnarray*}
	\chi(\mathcal{E}(td))&=&\chi(\mathcal{O}_{\mathbb{P}^2}^{\frac{r}{2}(d+1)}(d-1+td))-\chi(\mathcal{O}_{\mathbb{P}^2}^{\frac{r}{2}(d-1)}(d-2+td))\\
	&=& \frac{r}{2}(d+1)\chi(\mathcal{O}_{\mathbb{P}^2}(d-1+td))- \frac{r}{2}(d-1) \chi(\mathcal{O}_{\mathbb{P}^2}(d-2+td))\\
	&=& \frac{r}{2}(d+1) \binom{d-1+td+2}{2}-\frac{r}{2}(d-1) \binom{d-2+td+2}{2}\\
	&=& \frac{r}{4}[(d+1) ((t+1)d+1) ((t+1)d)- (d-1)((t+1)d) ((t+1)d-1)]\\
	&=& \frac{r}{4}((t+1)d)[(d+1) ((t+1)d+1) - (d-1)((t+1)d-1)]\\
	&=& \frac{r}{2}d^{2} (t+1)(t+2)=d^{2} r \binom{t+2}{2}.
\end{eqnarray*}

Next, we shall prove that $\mathcal{E}$ is ACM with respect to the polarization $dH$; i.e. $H^1(\mathcal{E}(td))=0$ for all $t\in \mathbb{Z}$.

If $t\geq -1$, then consider the long exact sequence obtained by twisting given sequence by $td$:
\[
 \ldots \to H^{1}(\mathcal{O}_{\mathbb{P}^2}^{\frac{r}{2}(d+1)}(d-1+td)) \to H^{1}(\mathcal{E}(td)) \to H^{2}(\mathcal{O}_{\mathbb{P}^2}^{\frac{r}{2}(d-1)}(d-2+td)) \to \ldots.
\]

Since $H^{1}(\mathcal{O}_{\mathbb{P}^2}^{\frac{r}{2}(d+1)}(d-1+td))=0$ for all $t\in \mathbb{Z}$ and $H^{2}(\mathcal{O}_{\mathbb{P}^2}^{\frac{r}{2}(d-1)}(d-2+td))=H^{0}(\mathcal{O}_{\mathbb{P}^2}^{\frac{r}{2}(d-1)}(2-d-td-3))=0$ for all $t\geq -1$, $H^{1}(\mathcal{E}(td))=0$ for all $t\geq -1$.

For $t=-2, -3, \ldots, - \alpha$, we have $H^{1}(\mathcal{E}(td))=0$ by assumption.

Now consider the short exact sequence obtained by twisting the given resolution by $-d$ and then restricting to a line $L \subset \mathbb{P}^2$:
\[
 0 \to \mathcal{O}_{L}^{\frac{r}{2}(d-1)}(-2) \to \mathcal{O}_{L}^{\frac{r}{2}(d+1)}(-1) \to \mathcal{E}(-d)|_{L} \to 0.
\]

By a well-known result of Grothendieck, $\mathcal{E}|_{L}= \bigoplus_{i=1}^{r} \mathcal{O}(a_{i})$ for a unique collection of integers $a_{1}\geq\ldots\geq a_{r}$. Since $H^{1}(\mathcal{O}_{L}(-1))=0$, it follows that $H^{1}(\mathcal{E}|_{L}(-d))=0$. Therefore, $a_{r}-d \geq -1$; i.e. $a_{r}\geq d-1$. Since
\[
 deg(\mathcal{E}|_{L})=\frac{r}{2}(d+1)(d-1)-\frac{r}{2}(d-1)(d-2)=\frac{3r}{2}(d-1),
\]
we have $\sum^{r}_{i=1}a_{i}=\frac{3r}{2}(d-1)$.

If we tensor the sequence
\[
 0 \to \mathcal{O}_{\mathbb{P}^2} (-1) \to \mathcal{O}_{\mathbb{P}^2} \to \mathcal{O}_{\mathbb{P}^2}|_{L} \to 0
\]
by $\mathcal{E}(-\alpha d)$, we obtain
\begin{equation}\label{eq:ses}
 0 \to \mathcal{E} (-\alpha d-1) \to \mathcal{E} (-\alpha d) \to \mathcal{E}|_{L} (-\alpha d) = \bigoplus^{r}_{i=1} \mathcal{O} (a_{i}- \alpha d) \to 0.
\end{equation}

If we were to assume $a_{1} - \alpha d \geq 0$, then we would have
\begin{eqnarray*}
 \sum^{r}_{i=1} a_{i} &\geq& \alpha d + (r-1)(d-1),\\
 \frac{3r}{2}(d-1) &\geq& \alpha d + (r-1)(d-1),\\
 dr &\geq& (2 \alpha -2)d+r+2,
\end{eqnarray*}
which is a contradiction. So, $a_{1}- \alpha d < 0$; and hence $a_{i}- \alpha d < 0$ for $i=1,\ldots, r$. Therefore, $H^{0}(a_{i}- \alpha d)=0$ for $i=1,\ldots, r$; i.e. $H^{0}(\mathcal{E}|_{L}(- \alpha d))=0$. We also have $H^{1}(\mathcal{E}(-\alpha d))=0$ by assumption. It follows that $H^{1}(\mathcal{E}(-\alpha d -1))=0$. Using negative twists of the sequence \ref{eq:ses} above and by induction, it can be easily shown that $H^{1}(\mathcal{E}(-\alpha d -k))=0$ for all $k\geq 0$. In particular, for integers $t \leq - \alpha -1$, we have $H^{1}(\mathcal{E}(t d))=0$. This proves that $\mathcal{E}$ is an ACM vector bundle with respect to the polarization $dH$.

Therefore, $\mathcal{E}$ is an Ulrich bundle on $(\mathbb{P}^2,dH)$ by \cite{CKM2013}*{Proposition 2.3}.
\end{proof}

\section{The Case $d=2$}
By Corollary \ref{cor:rank_and_degree}, the Ulrich bundles on $(\mathbb{P}^2, 2H)$ must have even rank. In this section, we prove that there is a unique Ulrich bundle of rank 2 on $(\mathbb{P}^2, 2H)$ and that all the other Ulrich bundles are direct sums of copies of this unique Ulrich bundle of rank 2.

\begin{lemma}\label{lemma:h_1_e_0}
Let $\mathcal{E}$ be an Ulrich bundle of rank 2 on $(\mathbb{P}^2, 2H)$. Then $\mathcal{E}$ is stable, and we have $h^1(\mathcal{E} \otimes \mathcal{E}^{\vee})=0$.
\end{lemma}
\begin{proof}
Since there are no Ulrich line bundles on $(\mathbb{P}^2, 2H)$ by Proposition \ref{prop:ulrich_line_bundles}, $\mathcal{E}$ is stable by \cite{CKM2012}*{Lemma 2.15}\footnote{We note that even though that lemma is stated for hypersurfaces, the proof applies verbatim to any projective variety.}.

By \cite{CHGS2012}*{Proposition 2.12}, we have $\chi(\mathcal{E} \otimes \mathcal{E}^{\vee})=1$. Now, since $\mathcal{E}$ is stable, it is simple and hence $h^0(\mathcal{E} \otimes \mathcal{E}^{\vee})=1$. Also,
\begin{align*}
 h^2(\mathcal{E} \otimes \mathcal{E}^{\vee}) &= h^0(\mathcal{E} \otimes  \mathcal{E}^{\vee} (-3)) \\
 &= hom(\mathcal{E}(3), \mathcal{E}) \\
 &= 0
\end{align*}
since the reduced Hilbert polynomial of $\mathcal{E}(3)$ is greater then the reduced Hilbert polynomial of $\mathcal{E}$ and both vector bundles are semistable. This implies $h^1(\mathcal{E} \otimes \mathcal{E}^{\vee})=0$.
\end{proof}

\begin{theorem}\label{thm:unique_rank_2_on_degree_2}
There is a unique Ulrich bundle $\mathcal{E}_0$ of rank 2 on $(\mathbb{P}^2, 2H)$.
\end{theorem}

\begin{proof}
By Corollary \ref{cor:short_exact_sequence}, an Ulrich bundle $\mathcal{E}$ of rank 2 on $(\mathbb{P}^2, 2H)$ fits into a short exact sequence
\[
 0 \to \mathcal{O}_{\mathbb{P}^2} \to \mathcal{O}_{\mathbb{P}^2}^{3}(1) \to \mathcal{E} \to 0.
\]
Now consider the vector space $V=Hom(\mathcal{O}_{\mathbb{P}^2}, \mathcal{O}_{\mathbb{P}^2}^{3}(1))$. The injective morphisms in $V$ with locally free cokernel form a nonempty and open, hence irreducible, subset by \cite{cB1991}*{Section 4.1}. This means that the moduli space of Ulrich bundles of rank 2 on $(\mathbb{P}^2, 2H)$ is irreducible. Now, Lemma \ref{lemma:h_1_e_0} immediately implies that this moduli space consists of a single point.
\end{proof}
\begin{remark}
By \cite{ESW2003}*{Proposition 5.9}, the tangent bundle $T$ on $(\mathbb{P}^2, 2H)$ is an Ulrich bundle. This theorem implies that it is the \emph{unique} one.
\end{remark}

\begin{theorem}
There is a unique Ulrich bundle of rank $2k$ on $(\mathbb{P}^2, 2H)$; and it is $\mathcal{E}_0^{\oplus k}$.
\end{theorem}
\begin{proof}
We use induction on $k$, the case $k=1$ being evident from Theorem \ref{thm:unique_rank_2_on_degree_2}.

Suppose now that $\mathcal{F}$ is an Ulrich bundle of rank $2k$ on $(\mathbb{P}^2, 2H)$ with $k \geq 2$; and suppose that the statement is proved for ranks smaller than $2k$. By Proposition \ref{prop:c_1}, we have $c_1(\mathcal{F})=3k$. Using \cite{CHGS2012}*{Proposition 2.12}, we obtain $\chi(\mathcal{F} \otimes \mathcal{F}^{\vee})=k^2$. Now, by Lemma \ref{lemma:h2_vanishing}, $h^2(\mathcal{F} \otimes \mathcal{F}^{\vee})=0$. Hence, $h^0(\mathcal{F} \otimes \mathcal{F}^{\vee}) - h^1(\mathcal{F} \otimes \mathcal{F}^{\vee}) = k^2$, which implies that $h^0(\mathcal{F} \otimes \mathcal{F}^{\vee}) \geq 4$. We conclude that $\mathcal{F}$ is not stable, and therefore it can be written as an extension of Ulrich bundles of lower rank
\[
 0 \to \mathcal{F}^{\prime} \to \mathcal{F} \to \mathcal{F}^{\prime\prime} \to 0
\]
by \cite{CKM2012}*{Lemma 2.15}\footnote{See footnote above.}.

Now since both $\mathcal{F}^{\prime}$ and $\mathcal{F}^{\prime\prime}$ are direct sums of copies of $\mathcal{E}_0$, and $h^1(\mathcal{E}_0 \otimes \mathcal{E}_0^{\vee})=0$ by Lemma \ref{lemma:h_1_e_0}, it follows that the short exact sequence above is split. Hence, $\mathcal{F}$ is a direct sum of copies of $\mathcal{E}_0$.
\end{proof}

\section{The Case $d \geq 3$}
In this section, we discuss existence of stable Ulrich bundles of any given rank $r \geq 2$ on $(\mathbb{P}^2,dH)$, subject to the restriction that $r$ must be even whenever $d$ is. The existence of rank 2 Ulrich bundles on $(\mathbb{P}^2,dH)$ was proved in the work of Eisenbud, Schreyer and Weyman (\cite{ESW2003}); note that these are automatically stable by Proposition \ref{prop:ulrich_line_bundles}. Using a method due to Casanellas and Hartshorne (\cite{CHGS2012}), we prove that for even $d$, there exist stable Ulrich bundles of all even ranks on $(\mathbb{P}^2,dH)$. This method uses consecutive extensions of Ulrich bundles of ranks 2, or ranks 2 and 3, and counting deformations to ensure that there are more simple Ulrich bundles than strictly semistable ones, hence ensuring that there exist \emph{stable} Ulrich bundles.

The existence of rank 3 Ulrich bundles on $(\mathbb{P}^2,dH)$ whenever $d$ is odd is a more difficult problem. (Again, note that rank 3 Ulrich bundles are necessarily stable.) We prove that, given the existence of rank 3 Ulrich bundles on $(\mathbb{P}^2,dH)$ with $d$ odd, there exist Ulrich bundles of all ranks $r \geq 2$ on $X$. In Section \ref{sec:appendix}, we include a \textit{Macaulay2} program to check existence of rank 3 Ulrich bundles; we are able to prove that Ulrich bundles of rank 3 exist for all odd degrees up to 43.

\begin{theorem}\label{thm:d_even}
Suppose that $d$ is even. Then there exist stable Ulrich bundles of rank $r=2k$ on $(\mathbb{P}^2,dH)$ for $k \geq 1$.
\end{theorem}
\begin{proof}
We use induction on $k$. The case $k=1$ follows from \cite{ESW2003}*{Proposition 5.9}. In this case, the dimension of first-order deformations of a rank 2 Ulrich bundle $\mathcal{E}$ is given by $h^1(\mathcal{E}^{\vee} \otimes \mathcal{E})$. Note first that $\mathcal{E}$ is stable, hence simple. Therefore, by Corollary \ref{cor:h1_simple} $h^1(\mathcal{E}^{\vee} \otimes \mathcal{E})=d^2-4$. Since $h^1(\mathcal{E}^{\vee} \otimes \mathcal{E}) > 0$ and $h^2(\mathcal{E}^{\vee} \otimes \mathcal{E}) = 0$, we conclude that there are more than one nonisomorphic rank 2 Ulrich bundle on $(\mathbb{P}^2,dH)$.

Now assume that the theorem is proved up to rank $2k-2$. Let $r=2k$ with $k \geq 2$. Choose an Ulrich bundle $\mathcal{E}_1$ of rank 2 and another stable Ulrich bundle $\mathcal{E}_2$ of rank $2k-2$ on $(\mathbb{P}^2,dH)$. Since $\mathcal{E}_1$ and $\mathcal{E}_2$ are necessarily non-isomorphic, $h^0(\mathcal{E}_1^{\vee} \otimes \mathcal{E}_2)=h^0(\mathcal{E}_1, \mathcal{E}_2)=0$ by \cite{CKM2013}*{Proposition 2.3} and \cite{HL2010}*{Proposition 1.2.7}. (In the case $k=2$, we must make sure that $\mathcal{E}_1$ and $\mathcal{E}_2$ are not isomorphic. This can be done in view of the fact that the first-order deformations of $\mathcal{E}_1$ have dimension $d^2-4 > 0$.) Similarly, we have $h^2(\mathcal{E}_1^{\vee} \otimes \mathcal{E}_2)=0$. Therefore, we have
\begin{align*}
 h^1(\mathcal{E}_1^{\vee} \otimes \mathcal{E}_2) &= - \chi(\mathcal{E}_1^{\vee} \otimes \mathcal{E}_2) \\
 &= (k-1)(d^2-5)
\end{align*}
as before. Since $h^1(\mathcal{E}_1^{\vee} \otimes \mathcal{E}_2) > 0$, there exist nonsplit extensions of $\mathcal{E}_1$ by $\mathcal{E}_2$; and these are simple by \cite{CHGS2012}*{Lemma 4.2}. Hence, there exist \emph{simple} Ulrich bundles of rank $r=2k$ on $(\mathbb{P}^2,dH)$. As above, we can compute the dimension of first-order deformations of such an Ulrich bundle $\mathcal{E}$; we obtain $h^1(\mathcal{E}^{\vee} \otimes \mathcal{E})=k^2(d^2-5)+1$ by Corollary \ref{cor:h1_simple}.

We must now compute an upper bound for the dimension of \emph{strictly semistable} Ulrich bundles of rank $2k$ and prove that it is smaller than $k^2(d^2-5)+1$. Following the explanation under \cite{CHGS2012}*{Remark 4.6}, it is enough to find an upper bound for the dimension of Ulrich bundles of rank $2k$ that are obtained as extensions of rank 2 and rank $2k-2$ stable Ulrich bundles. These have moduli spaces of dimensions $d^2-4$ and $(k-1)^2(d^2-5)+1$ respectively by Corollary \ref{cor:h1_simple}. And the extension space of such bundles has dimension $(k-1)(d^2-5)-1$. It is now trivial to verify that
\[
 (d^2-4) + ((k-1)^2(d^2-5)+1) + (k-1)(d^2-5) -1 < k^2(d^2-5)+1.
\]
\end{proof}

\begin{theorem}\label{thm:d_odd}
Suppose that $d$ is odd, and there exists a rank 3 Ulrich bundle on $(\mathbb{P}^2,dH)$. Then there exist stable Ulrich bundles of any rank $r \geq 2$ on $(\mathbb{P}^2,dH)$.
\end{theorem}
\begin{proof}
The proof of this theorem follows the same outline as the proof of Theorem \ref{thm:d_even}, so we omit many of the details.

Again we use induction on the rank $r$. The case $r=2$ follows from \cite{ESW2003}*{Proposition 5.9} and the case $r=3$ is included in the hypothesis. Suppose now that $r \geq 4$. We distinguish two cases.

\textit{Case 1: $r=2k$, $k \geq 2$.} In this case, it can be proved that there exist \emph{simple} Ulrich bundles of rank $r$ by the argument in the proof of Theorem \ref{thm:d_even}; and the dimension of the first-order deformations of such an Ulrich bundle is $k^2(d^2-5)+1$ by Corollary \ref{cor:h1_simple}.

To prove the existence of \emph{stable} Ulrich bundles of rank $r$, we must compute the dimension of \emph{strictly semistable} Ulrich bundles of rank $r$ and prove that it is smaller than $k^2(d^2-5)+1$. We consider only extensions of rank 2 and rank $r-2$ stable Ulrich bundles, or extensions of rank 3 and rank $r-3$ stable Ulrich bundles.

For a rank 2 Ulrich bundle $\mathcal{E}$ and a rank $r-2$ stable Ulrich bundle $\mathcal{F}$, we have
\[
 h^1(\mathcal{E}^{\vee} \otimes \mathcal{F}) = (k-1)(d^2-5).
\]
For a rank 3 Ulrich bundle $\mathcal{E}$ and a rank $r-3$ stable Ulrich bundle $\mathcal{F}$, we have
\[
 h^1(\mathcal{E}^{\vee} \otimes \mathcal{F}) = \frac{3}{4}(r-3)(d^2-5).
\]
(In this last case, we must make sure that $\mathcal{E}$ and $\mathcal{F}$ are not isomorphic if $r=6$. This can be done in view of the fact that the first-order deformations of $\mathcal{E}$ have dimension $1+\frac{9}{4}(d^2-5) > 0$.)

Both these numbers are positive, hence there exist nonsplit extensions of $\mathcal{E}$ by $\mathcal{F}$, which are simple by \cite{CHGS2012}*{Lemma 4.2}.

Now we must prove that the  the dimension of \emph{strictly semistable} Ulrich bundles of rank $r$ is strictly smaller than the dimension of simple Ulrich bundles of rank $r$, which is $k^2(d^2-5)+1$.

In the case of extensions of rank 2 and rank $r-2$ stable Ulrich bundles, we have to prove that
\[
 (d^2-4) + \frac{1}{4}(4+(2k-2)^2(d^2-5)) + (k-1)(d^2-5) -1 < k^2(d^2-5)+1.
\]
In the case of extensions of rank 3 and rank $r-3$ stable Ulrich bundles, we have to prove that
\[
 1 + \frac{9}{4} (d^2-5) + 1 + \frac{(r-3)^2}{4}(d^2-5) + \frac{3}{4}(r-3)(d^2-5) -1 < k^2(d^2-5) +1.
\]
It is a straightforward exercise to verify both inequalities.

\textit{Case 2: $r=2k+1$, $k \geq 2$.} The proof is almost identical to the proof of the second half of case 1 above. For an Ulrich bundle $\mathcal{E}$ of rank 3 and a stable Ulrich bundle $\mathcal{F}$ of rank $r-3$, we have
\[
 h^1(\mathcal{E}^{\vee} \otimes \mathcal{F}) = \frac{3}{4}(r-3)(d^2-5).
\]
Proving the existence of stable rank $r$ Ulrich bundles now comes down to verifying that
\[
 1 + \frac{9}{4} (d^2-5) + 1 + \frac{(r-3)^2}{4}(d^2-5) + \frac{3}{4}(r-3)(d^2-5) -1 < k^2(d^2-5) +1,
\]
which is again straightforward.
\end{proof}

\appendix
\section{A \textit{Macaulay2} Program To Check Existence Of Rank 3 Ulrich Bundles}\label{sec:appendix}
For a given odd $d$, the following \textit{Macaulay2} code can be used to verify that there exist linear maps
\[
 \mathcal{O}_{\mathbb{P}^2}^{\frac{3}{2}(d-1)}(d-2) \to \mathcal{O}_{\mathbb{P}^2}^{\frac{3}{2}(d+1)}(d-1)
\]
whose cokernels $\mathcal{E}$ satisfy the vanishing conditions $H^1(\mathcal{E}(-2d))=H^1(\mathcal{E}(-3d))=0$. Note that we do not check whether $\mathcal{E}$ is locally free; since the maps whose cokernels are locally free form a dense subset in the corresponding vector space.
\\

\begin{verbatim}
i1 : kk=ZZ/32003; R=kk[x,y,z]; d=7;

i2 : p=3*(d+1)//2; q=3*(d-1)//2; 

i3 : M=random(R^p, R^{q:-1}); E=(sheaf (cokernel M))(d-1);

             12       9
o3 : Matrix R   <--- R

i4 : (HH^1(E(-2*d)), HH^1(E(-3*d)))

o4 = (0, 0)

o4 : Sequence
\end{verbatim}

We have used this code to verify the existence of rank 3 Ulrich bundles on $(\mathbb{P}^2,dH)$ for odd degrees up to $d=43$. Therefore, we can make the following conjecture.
\begin{conjecture}
There exists a rank 3 Ulrich bundle on $(\mathbb{P}^2,dH)$ for all odd degrees $d \geq 3$.
\end{conjecture}

\end{document}